\documentclass[12pt,twoside, final]{amsart}
\usepackage{amsmath,amsthm,amscd,amsfonts,amssymb,enumerate}
\usepackage{graphicx}
\usepackage{color}
\usepackage[colorlinks]{hyperref}

\newtheorem{theorem}{Theorem}[section]
\newtheorem{proposition}[theorem]{Proposition}
\newtheorem{lemma}[theorem]{Lemma}
\newtheorem{corollary}[theorem]{Corollary}
\theoremstyle{definition}
\newtheorem{definition}[theorem]{Definition}
\newtheorem{example}[theorem]{Example}
\theoremstyle{remark}

\numberwithin{equation}{section}

\begin{document}

\title[]{Harmonicity of unit vector fields with respect to a class of Riemannian metrics}

\author{A. Baghban and E. Abedi}
\address[]{Department of Mathematics, Azarbaijan Shahid Madani University, Tabriz 53751 71379, I. R. Iran.}

\email{amirbaghban@aut.ac.ir,\ \ esabedi@azaruniv.ac.ir}

\maketitle

\begin{abstract}
The isotropic almost complex structures induce a Riemannian metric $g_{\delta , \sigma }$ on $TM$, which are the generalized type of Sasakian metric. In this paper, the Levi-Civita connection of $g_{\delta , \sigma }$ is calculated and the harmonicity of unit vector fields from $(M,g)$ to $(S(M), i^*g_{\delta , 0})$ is investigated, where $i^*g_{\delta , 0}$ is a particular type of induced metric $i^* g_{\delta ,\sigma}$. Finally, an important example is presented which satisfies in main theorem of the paper.
\end{abstract}

\section{Introduction}

�
Let $(M,g)$ be a Riemannian manifold and $(TM,g_s)$ be its tangent bundle equipped with Sasaki metric. Moreover, suppose $(S(M),i^*g_s)$ is the unit tangent bundle of $(M,g)$ with induced Sasaki metric. Denote by $\Gamma TM$ the set of all smooth vector fields on $M$.

Since, every vector field defines a map from $(M,g)$ to $(TM,g_s)$, it is natural to investigate the harmonicity of vector fields as a map with the exception that the energy functional is restricted to the vector fields on $M$ instead of all functions from $(M,g)$ to $(TM,g_s)$. Medrano \cite{gil} investigated the harmonicity of a vector field from $(M,g)$ to $(TM,g_s)$ and he proved that the parallel vector fields are the only ones when $(M,g)$ is a compact Riemannian manifold.

One can study the harmonicity of vector fields using by variational problem. This links the tension tensor field and critical points of the energy functional. Ishihara \cite{ishihara} calculated the tension tensor field of a vector field as a map from $(M,g)$ to $(TM,g_s)$ and represented another equivalent to the harmonicity of them. He showed that, the necessary and sufficient conditions for the harmonicity of a vector field is the vanishing of its Laplacian, i.e., $ \Delta _g X=0$.

On a compact Riemannian manifold $(M,g)$, the conditions $\nabla X=0$ and $\Delta _g X=0$ are equivalent for an arbitrary vector field, so Medrano and Ishihara present two different Equivalences for harmonicity of an arbitrary vector field.

We know that, if we restrict the energy functional to the unit vector fields, the vanishing of $ \Delta _g X$ ensure the harmonicity of unit vector fields as a map from compact Riemannian manifold $(M,g)$ to $(S(M),i^*g_s)$. However, this is a big condition for a unit vector field to be a harmonic vector field, therefore, it is natural to investigate the harmonicity of unit vector fields as a map from $(M,g)$ to $(S(M),i^*g_s)$. Wiegmink \cite{Wiegmink} demonstrated that a unit vector field $X$ is a harmonic unit vector field if and only if $\Delta _g X=||\nabla X||^2 X$.

The contribution on the harmonicity of vector fields did not limited to the tangent bundles equipped with Sasaki metric. Abbassi and Calvaruso and Perrone \cite{abbasi} considered the problem of determining which vector fields $X:(M,g) \longrightarrow (TM,G)$ define harmonic maps where $G$ is an arbitrary g-natural metric on $TM$.

Aguilar \cite{aguilar} introduced Isotropic almost complex structures. Dragomir and Perrone \cite{dragomir} introduced the problem of studing the harmonic (unit) vector fields where $TM$ equipped with Riemannian metric $g_{\delta , \sigma}$ induced by an arbitrary isotropic almost complex structure $J_{\delta , \sigma }$. In this paper we solve this problem when the unit tangent bundle $S(M)$ is equipped with induced Riemannian metric $i^*g_{\delta , 0}$, which is a particular type of $i^*g_{\delta ,\sigma}$.

Section 2, gives a plenary preliminarie of the tangent bundle, unit tangent bundle, pullback tangent bundle, energy functional, variational problem and isotropic almost complex structures. Whereas Section 3 presents a discussion about the induced metrics $g_{\delta , \sigma}$ on $TM$  and calculating of its Levi-civita connection. In section 4 we calculate the tension tensor field of an arbitrary vector field $X:(M,g) \longrightarrow (TM,g_{\delta ,\sigma})$ and section 5 is devoted to calculate the tension tensor field of unit vector fields from $(M,g)$ to $(S(M),i^*g_{\delta ,\sigma})$. In this section the critical points of the energy functional are determined.
\section{Preliminaries}
\subsection{Brief discussion about the tangent bundle and pullback tangent bundle}
Let $(M,g)$ be an n-dimensional Riemannian manifold and $\nabla$ its Levi-Civita connection. Moreover let $\pi: TM \longrightarrow M$ be its tangent bundle and $K:TTM \longrightarrow TM$ be connection map with respect to $\nabla$. We can split $TTM$ to vertical and horizontal sub-vector bundles $\mathcal{V}$ and $\mathcal{H}$, respectively, i.e., for every $v \in TM$, $T_vTM=\mathcal{V}_v  \oplus \mathcal{H}_v$. These distributions have the following properties
\begin{itemize}
\item $ \pi_{*v}\mid_{\mathcal{H} _v}: {\mathcal{H}}_v \longrightarrow T_{\pi(v)}M $ is an isomorphism.
\item $ K \mid_{{\mathcal{V}}_v }: \mathcal{V}_v \longrightarrow T_{\pi(v)} M $ is an isomorphism.
\end{itemize}

$X \in \Gamma (TM)$ has vertical and horizontal lifts $(X^v)_u=(K \mid_{{\mathcal{V}}_u })^{-1}X_{\pi (u)} \in \mathcal{V}_u$ and $(X^h)_u=(\pi_{*v}\mid_{\mathcal{H} _u})^{-1}X_{\pi (u)}  \in \mathcal{H}_u$, when $X_{\pi(u)} \in T_{\pi (u)}M$.

Let $(p,u) \in TM$, the Lie bracket of the horizontal and vertical vector fields are expressed as follows
\begin{align}
[X^h,Y^h]_{(p,u)}&=[X,Y]^h_{(p,u)}-(R(X,Y)u)^v_{(p,u)},\\
[X^h,Y^v]_{(p,u)}&=(\nabla _X Y)^v_{(p,u)},\\
[X^v,Y^v]_{(p,u)}&=0_{(p,u)}.
\end{align}
Let $f=(f_1,...,f_n):(M,g) \longrightarrow (N,h)$ be a $ C^{\infty} $ function between Riemanian manifolds and $\nabla ^M$ and $\nabla ^N$ be Levi-Civita connections of $g$ and $h$, respectively. If $TN$ be the tangent bundle of $N$, then we can define vector bundle $f^{-1}TN$ on $M$ by
\begin{align*}
(f^{-1} TN)_p=T_{f(p)}N \hspace{10mm} \forall p \in M.
\end{align*}
There is a natural connection on $f^{-1} TN$ defined by
\begin{align*}
(f^{-1} \nabla ^N )_{\frac{ \partial}{ \partial x^i}}Y_{\beta}=\frac{\partial f^{\alpha}}{\partial x^i}[(\Gamma ^N)^{\gamma}_{\alpha \beta } o f]Y_{\gamma} \quad \forall i=1,...,m \quad and \quad \alpha , \beta , \gamma =1,...,n,
\end{align*}
where $( x^1,...,x^m)$ be a local chart on $M$ and $( y^1,...,y^n)$ be a local chart on $N$ and $\lbrace Y_1= \frac{ \partial}{ \partial y^1}o f,...,Y_n= \frac{ \partial}{ \partial y^n}o f  \rbrace$ is a local basis of sections on $f^{-1}TN$ and $\lbrace (\Gamma ^N)^{\gamma}_{\alpha \beta }\rbrace$ are the Christoffel symbols of the Levi-Civita connection of $h$.

\subsection{Energy functional on vector fields}
Let $(M,g)$ be an n-dimensional compact manifold and $(TM,G)$ be its tangent bundle equipped with an arbitrary Riemannian metric $G$. Moreover, let $F:(M,g) \longrightarrow (TM,G)$ be an arbitrary smooth function. Then, the energy of $F$ is defined by 
\begin{align*}
E(F)=\frac{1}{2} \int _M ||dF||^2 dvol(g),
\end{align*}
where $||dF||$ is the Hilbert-Schmitd norm of $dF$, i.e., $||dF||=tr_g(F^*G)$ and $dvol(g)$ is the Riemannian volume form on $M$. We call $E$ the energy functional.

In particular, we can restrict the energy functional to the sections of $TM$, that is, all of the vector fields on $M$. Due to this, we call $E$ the energy functional on vector fields.
\subsection{Variational problems on vector fields}
Given an arbitrary tangent vector field $X \in \Gamma  TM$. A smooth 1-parameter variation on $X$ is a map $\mathcal{U}: M \times (-\delta , \delta ) \longrightarrow TM$ with this feature that $ \mathcal{U}_t $ is a vector field on $M$ defined by $ \mathcal{U}_t (p)=\mathcal{U}(p,t) \in T_pM$ , for every $t \in (-\delta , \delta)$ and $\mathcal{U}(p,0)=X_p$. We can define a vector field on $\mathcal{U} (M \times \lbrace 0 \rbrace )= X(M) \subseteq TM$ defined by $\mathcal{V}_{\mathcal{U}(p,0)}= \frac{d}{dt}|_{t=0}\mathcal{U} (p,t)$. It is clear that $\mathcal{V}$ is a vertical vector field and we call it variational vector field.

Variational problem on vector fields calculates the critical points of $E$, hearafter called as harmonic vector fields.

\subsection{Isotropic almost complex structures}
Isotropic almost complex structures are a generalized type of natural almost complex structure on $TM$ due to Aguilar \cite{aguilar}. The isotropic almost complex structures are determining an almost kahler metric
whose kahler 2-form is the pullback of the canonical symplectic form on $T^*M$ to $TM$ via $\mathfrak{b} :TM \longrightarrow T^*M$. Moreover, Aguilar \cite{aguilar} proved that there is an isotropic complex structure on $M$ if and
only if $M$ is of constant sectional curvature.

\begin{definition}\cite{aguilar}
An almost complex structure $J$ on $TM$ is said to be isotropic with respect to the
Riemannian metric on $M$, if there are smooth functions $\alpha, \delta, \sigma :TM \longrightarrow R$ such that $\alpha \delta - \sigma^2 =1$ and
\begin{align}
 JX^h=\alpha X^v + \sigma X^h, \hspace{20mm} JX^v=-\sigma X^v - \delta X^h.
\end{align}
\end{definition}
\section{A class of Riemannian metrics on $TM$}
Let $\Theta \in \Omega ^1 (TM)$ be a 1-form on $TM$ defined by
\begin{align}
\Theta _v (A)=g_{\pi(v)}(\pi _*(A),v), \hspace{5mm}  A \in T_vTM, \hspace{5mm}  v \in TM.
\end{align}
\begin{definition} \cite{dragomir} Let $J_{\delta,\sigma}$ be an isotropic almost complex structure. A (0,2)-tensor $g_{\delta,\sigma}(A,B)=d\Theta (J_{\delta,\sigma}A,B)$ defines a Riemannian metric on $TM$ provided that $0 \lvertneqq \alpha $ where $ A,B \in \Gamma TTM$.
\end{definition}
A simple calculation shows that 
\begin{align}
g_{\delta , \sigma}(X^h,Y^h)=& \alpha g(X,Y) o\pi ,\\
g_{\delta , \sigma}(X^h,Y^v)=& - \sigma g(X,Y) o\pi ,\\ 
g_{\delta , \sigma}(X^v,Y^v)=& \delta g(X,Y) o\pi .
\end{align}
For calculating the Levi-Civita connection of $g_{\delta ,\sigma}$ we need the following lemma.
\begin{lemma}\cite{sarih} Let $X$, $Y$ and $Z$ be any vector field on $M$. Then
\begin{align}
&X^h(g(Y,Z)o\pi)=(Xg(Y,Z))o\pi ,\\
&X^v(g(Y,Z)o\pi)=0.
\end{align}
\end{lemma}

\begin{theorem} Let $g_{\delta,\sigma}$ be a Riemannian metric on $TM$ as above. Then the Levi-Civita connection ̅$\bar{\nabla}$ of $g_{\delta,\sigma}$ at $(p,u) \in TM$ is given by

\begin{align}
\bar{\nabla}_{X^h} Y^h&=(\nabla _X Y)^h- \frac{\sigma}{\alpha} (R(u,X)Y)^h+ \frac{1}{2 \alpha} X^h (\alpha) Y^h +  \frac{1}{2 \alpha}Y^h(\alpha) X^h
\nonumber \\
& - \frac{\sigma}{\delta} (\nabla _X Y)^v-\frac{1}{2}(R(X,Y)u)^v- \frac{1}{2 \delta} X^h(\sigma) Y^v
\nonumber \\
&-\frac{1}{2\delta}Y^h(\sigma)X^v-\frac{1}{2}g(X,Y)\bar{\nabla} \alpha ,\\
\bar{\nabla}_{X^h} Y^v&=-\frac{\sigma}{\alpha} (\nabla _X Y)^h + \frac{\delta}{2 \alpha}(R(u,Y)X)^h-\frac{1}{2\alpha}X^h(\sigma) Y^h+\frac{1}{2\alpha}Y^v(\alpha) X^h
\nonumber\\
&+(\nabla _X Y)^v+\frac{1}{2\delta}X^h(\delta)Y^v- \frac{1}{2\delta}Y^v(\sigma)X^v+ \frac{1}{2}g(X,Y) \bar{\nabla} \sigma ,\\
\bar{\nabla}_{X^v} Y^h&=\frac{\delta}{2\alpha}(R(u,X)Y)^h+ \frac{1}{2\alpha}X^v(\alpha)Y^h-\frac{1}{2\alpha}Y^h(\sigma)X^h
\nonumber\\
&- \frac{1}{2\delta}X^v(\sigma)Y^v+\frac{1}{2\delta}Y^h(\delta)X^v+\frac{1}{2}g(X,Y)\bar{\nabla}\sigma ,\\
\bar{\nabla}_{X^v} Y^v&=-\frac{1}{2\alpha}X^v(\sigma)Y^h- \frac{1}{2\alpha}Y^v(\sigma)X^h+ \frac{1}{2\delta}X^v(\delta)Y^v+\frac{1}{2\delta}Y^v(\delta)X^v
\nonumber \\
&-\frac{1}{2}g(X,Y)\bar{\nabla}\delta .
\end{align}
\end{theorem}
\begin{proof} We just prove (14), the remaining ones are similar. Using Koszul formula, we have
\begin{align*}
2g_{\delta,\sigma}(\bar{\nabla}_{X^h} Y^h,Z^h)&=X^hg_{\delta,\sigma}(Y^h,Z^h)+Y^hg_{\delta,\sigma}(X^h,Z^h)-Z^hg_{\delta,\sigma}(X^h,Y^h)
\nonumber \\
&+g_{\delta,\sigma}([X^h,Y^h],Z^h)+g_{\delta,\sigma}([Z^h,X^h],Y^h)\\
&-g_{\delta,\sigma}([Y^h,Z^h],X^h).
\nonumber \\
\end{align*}
Using (1), (7) and (12) gives us
\begin{align*}
2g_{\delta,\sigma}(\bar{\nabla}_{X^h} Y^h,Z^h)&=X^h(\alpha)g(Y,Z)+ \alpha Xg(Y,Z)+Y^h(\alpha)g(X,Z)\\
&+\alpha Yg(X,Z)-Z^h(\alpha)g(X,Y)-\alpha Zg(X,Y)+\alpha g([X,Y],Z)\\
&+\sigma g(R(X,Y)u,Z)+ \alpha g([Z,X]Y)+\sigma g(R(Z,X)u,Y)\\
&-\alpha g([Y,Z],X)- \sigma g(R(Y,Z)u,X).
\end{align*}
Using the properties of the Levi-Civita connection of $g$ and the compatibility of $g$ with it, one can get
\begin{align*}
2g_{\delta,\sigma}(\bar{\nabla}_{X^h} Y^h,Z^h)&=g(X^h(\alpha)Y,Z)+g(Y^h(\alpha)X,Z)-Z^h(\alpha) g(X,Y)\\
&+2\alpha g(\nabla _XY,Z)+ \sigma g(R(X,Y)u,Z)+ \sigma g(R(Z,X)u,Y)\\
&-\sigma g(R(Y,Z)u,X).
\end{align*}
Taking into account (7) and the Bianchi’s first identity, we have
\begin{align*}
2g_{\delta,\sigma}(\bar{\nabla}_{X^h} Y^h,Z^h)&=g_{\delta,\sigma}(\frac{1}{\alpha}X^h(\alpha)Y^h+\frac{1}{\alpha}Y^h(\alpha)X^h-g(X,Y)\bar{\nabla}\alpha +2 (\nabla _X Y)^h\\
&-\frac{2\sigma}{\alpha}(R(u,X)Y)^h,Z^h),
\end{align*}
so the horizontal component of $\bar{\nabla}_{X^h} Y^h$ is
\begin{align*}
h(\bar{\nabla}_{X^h} Y^h)&=\frac{1}{2\alpha}X^h(\alpha)Y^h+\frac{1}{2\alpha}Y^h(\alpha)X^h-\frac{1}{2}g(X,Y)h(\bar{\nabla}\alpha) + (\nabla _X Y)^h\\
&-\frac{\sigma}{\alpha}(R(u,X)Y)^h,
\end{align*}
where $\bar{\nabla}\alpha=h(\bar{\nabla}\alpha)+v(\bar{\nabla}\alpha)$ is the spliting of the gradient vector field of $\alpha$ with respect to $g_{\delta,\sigma}$ to horizontal and vertical components, respectively. Similarly the vertical component of $\bar{\nabla}_{X^h} Y^h$ is
\begin{align*}
v(\bar{\nabla}_{X^h} Y^h)&=-\frac{1}{2\delta}X^h(\sigma)Y^v-\frac{1}{2\delta}Y^h(\sigma)X^v-\frac{1}{2}g(X,Y)v(\bar{\nabla}\alpha)-\frac{\sigma}{\delta}(\nabla _X Y)^v\\
&-\frac{1}{2}(R(X,Y)u)^v.
\end{align*}
Using the equation $(\bar{\nabla}_{X^h} Y^h)=h(\bar{\nabla}_{X^h} Y^h)+v(\bar{\nabla}_{X^h} Y^h)$, the proof is completed.
\end{proof}

\section{Tension tensor field of $X:(M,g)\longrightarrow (TM,g_{\delta,\sigma})$}
This section is devoted to calculate the tension tensor field of $X:(M,g)\longrightarrow (TM,g_{\delta,\sigma})$. The methods and the basic definitions and lemmas that we will use, can be found in \cite{dragomir}.

According to definition, the tension tensor field of $X$ is given by
\begin{align}
\tau (X)=tr_g(\beta),
\end{align}
where $\beta \in \Gamma(T^*M\otimes T^*M \otimes X^{-1}TTM)$ defined by
\begin{align}
\beta(Z,W)=(X^{-1}\bar{\nabla})_ZX_*W-X_*(\nabla _ZW) \hspace{5mm} Z,W \in \Gamma (TM),
\end{align}
the so called the second fundamental form of $X$.
 
 For determining of the expression of the tension tensor field of $X$, we begin with some lemmas.
 
 \begin{lemma}\cite{dragomir} let $V,X:M\longrightarrow TM$ be any smooth vector field. Then
 \begin{align}
 (\nabla_{\frac{\partial}{\partial x^i}}X)^k=\nabla _i \lambda ^k=\frac{\partial \lambda^k}{\partial x^i}+N_i^k, \hspace{4mm} X_*V=V^h+(\nabla _VX)^v=V^i(\partial _i+\frac{\partial \lambda ^j}{\partial x^i} \dot{\partial}_j),
 \end{align}
 where $ (\nabla_{\frac{\partial}{\partial x^i}}X)^k$ is the k-th component of covariant derivative of $X=\lambda ^j \frac{\partial}{\partial x^j}$ along $\frac{\partial}{\partial x^i}$ and $X_*$ is the derivative of $X$ as a map from $M$ to $TM$.
 \end{lemma}
 with a simple calculation, we have
\begin{align}
&\bar{\nabla}_{V^h}V^h=V(V^j) \delta_j + V^iV^j \lbrace \bar {\nabla}_{\partial _i} \partial _j-N_i^k \bar{ \nabla}_{\dot{ \partial}_k} \partial _j -N_j^l  \bar {\nabla} _{ \partial _i} \dot{ \partial }_l 
\nonumber \\
&+ (N_i^k \Gamma_{jk}^l- \lambda ^k \frac{\partial \Gamma_{jk}^l }{ \partial x^i }) \dot{\partial}_l +N_i^kN_j^l  \bar{\nabla}_{\dot{\partial}_k} \dot{\partial}_l \rbrace ,\\
&\bar{\nabla}_{V^h}(\nabla_VX)^v=V^i\frac{\partial}{\partial x^i}(\nabla _VX)^j \dot{\partial}_j +V^i(\nabla _VX)^j(\bar{\nabla}_{\partial _i } \dot{\partial}_j - N_i^k \bar{\nabla}_{\dot{\partial _k}} \dot{\partial}_j ),\\
&\bar{\nabla}_{(\nabla_VX)^v}V^h=V^i(\nabla _{ \partial ^i }X)^k V^j \lbrace { \bar{\nabla}_{\dot { \partial }_k} \partial _j - \Gamma _{jk} ^l \dot{\partial}_l- N_j ^l \bar{\nabla}_{\dot{{\partial}_k}} \dot{\partial _l} } \rbrace ,\\
&\bar{\nabla}_{(\nabla_VX)^v} (\nabla_VX)^v=(\nabla_VX)^k(\nabla_VX)^l \bar{\nabla}_{\dot{{\partial}_k}} \dot{\partial _l},
\end{align}
equation (21) is because $ \dot{ \partial }_i(\nabla_VX)^k=0$.

To calculate the tension tensor field of $X$, we need to define an important lemma specified as lemma 4.2. 
\begin{lemma}Let $X=\lambda ^i \frac{\partial}{\partial x^i }$ and $V=V^i \frac{\partial}{\partial x^i }$ be smooth vector fields on $M$. Then
\begin{align}
(X^{-1} \bar{\nabla})_VX_*V&=\bar{\nabla}_{V^h}V^h+\bar{\nabla}_{V^h}(\nabla _V X)^v
\nonumber \\
&+\bar{\nabla}_{(\nabla _V X)^v} V^h +\bar{\nabla}_{(\nabla _V X)^v}(\nabla _V X)^v.
\end{align}
\end{lemma}
\begin{proof} Using (17) gives us
\begin{align*}
(X^{-1} \bar{\nabla})_VX_*V=(X^{-1} \bar{\nabla})_V V^i(\partial _i +\frac{\partial \lambda ^j}{\partial x^i} \dot{\partial}_j).
\end{align*}
According to the deffinition of covariant derivative, one can write
\begin{align*}
(X^{-1} \bar{\nabla})_VX_*V&=V(V^i)(\partial _i + \frac{\partial \lambda ^j}{\partial x^i} \dot{\partial}_j)+V^iV^j (X^{-1} \bar{\nabla})_{\frac{\partial}{\partial x^j}}(\partial _i+ \frac{\partial \lambda ^k}{\partial x^i} \dot{\partial}_k).
\end{align*}
Using the diffinition of $X^{-1}\bar{\nabla}$ and the equation $\delta _i= \partial _i + N_i ^j \dot{\partial}_j$, one can write 
\begin{align}
(X^{-1} \bar{\nabla})_VX_*V&=V(V^j) \delta _j +V^k \frac{\partial V^i}{\partial x^k}(N_i^j + \frac{\partial \lambda ^j}{\partial x^i})\dot{\partial}_j + V^i V^j \frac{\partial ^2 \lambda ^k}{\partial x^i \partial x^j} \dot{\partial}_k 
\nonumber \\
&+V^iV^j \lbrace \bar{\nabla}_{\partial _j} \partial _i + \frac{\partial \lambda ^k}{\partial x^j} \bar{\nabla}_{\dot{\partial}_k} \partial _i + \frac{\partial \lambda ^k}{\partial x^i} \bar{\nabla}_{\partial _j} \dot{\partial}_k + \frac{\partial \lambda ^s}{\partial x^j} \frac{\partial \lambda ^k}{\partial x^i} \bar{\nabla}_{\dot{\partial}_s} \dot{\partial}_k \rbrace .
\end{align}
With adding/subtracting $V^iV^j \frac{\partial}{\partial x^j}(\Gamma_{il}^k \lambda ^l) \dot{\partial }_k$ in right hand side of equation (23), one can get
\begin{align*}
(X^{-1} \bar{\nabla})_VX_*V&=V(V^j)\delta _j + V^k \frac{\partial V^i}{\partial x^k}(\nabla _i \lambda ^j) \dot{\partial}_j +V^iV^j \frac{\partial}{\partial x^j} (N_i^k+\frac{\partial \lambda ^k}{\partial x^i}) \dot{\partial}_k\\
&+V^iV^j \lbrace \bar{\nabla}_{\partial _i} \partial _j + (\nabla _j \lambda ^k ) \bar{\nabla}_{\dot{\partial}_k} \partial _i -N_j^k \bar{\nabla}_{\dot{\partial}_k} \partial _i +(\nabla _i \lambda ^k) \bar{\nabla}_{\partial _j} \dot{\partial}_k \\
&-N_i^k \bar{\nabla}_{\partial _j} \dot{\partial}_k +\frac{\partial \lambda ^s}{\partial x^j} \frac{\partial \lambda ^k}{\partial x^i} \bar{\nabla}_{\dot{\partial}_s} \dot{\partial}_k \rbrace ,
\end{align*}
where we used the term $N_i^k=\Gamma _{il}^k \lambda ^l$. Using the equation$ \frac{\partial}{\partial x^j} (\Gamma_{il}^k \lambda ^l )=\lambda ^l \frac{\partial \Gamma _{il}^k}{\partial x^j}+\Gamma _{il}^k \frac{\partial \lambda ^l}{\partial x^j}$ and (17), we have
\begin{align*}
(X^{-1} \bar{\nabla})_VX_*V&=V(V^j) \delta _j +V^iV^j \lbrace \bar{\nabla}_{\partial _i} \partial _j -N_j^k \bar{\nabla}_{\partial _i} \dot{\partial}_k - N_i^k\bar{\nabla}_{\dot{\partial}_k} \partial _j \rbrace \\
& -V^iV^j\lambda ^l \frac{\partial \Gamma_{il}^k}{\partial x^j} \dot{\partial}_k +V^iV^j\Gamma _{il}^kN_j^l \dot{\partial}_k -V^iV^j \Gamma _{il}^k( \nabla _j \lambda ^l) \dot{\partial}_k\\
&+V^j\frac{\partial V^i}{\partial x^j}(\nabla _i \lambda ^k) \dot{\partial}_k + V^iV^j \frac{\partial}{\partial x^j}(\nabla _i \lambda ^k) \dot{\partial}_k +V^iV^j\lbrace (\nabla _j \lambda ^k) \bar{\nabla}_{\dot{\partial}_k} \partial _i\\
&(\nabla _i \lambda ^k) \bar{\nabla}_{\partial _j} \dot{\partial}_k +\frac{\partial \lambda ^s}{\partial x^j} \frac{\partial \lambda ^k}{\partial x^i} \bar{\nabla}_{\dot{\partial}_s} \dot{\partial}_k\rbrace .
\end{align*}
Substituting (18), (19), (20) and (21) in above equation completes the proof.
\end{proof}
For calculating the tension tensor field of $X$, we need to calculate the second fundamental form of $X$ at $(V,V)$. 
\begin{align}
\beta _X (V,V)&=(X^{-1} \bar{\nabla})_VX_*V - X_*(\nabla _V V)
\nonumber \\
&=\frac{1}{\alpha} \lbrace ( V^h(\alpha ) + (\nabla _V X)^v (\alpha ))V - \sigma R(X,V)V-(V^h(\sigma )
\nonumber \\
&+ (\nabla _V X)^v (\sigma )) \nabla _V X - \sigma \nabla _V \nabla _V X + \delta R(X,\nabla _V X)V  \rbrace ^h
\nonumber \\
&+\frac{1}{\delta} \lbrace -(V^h(\sigma ) +(\nabla _VX)^v (\sigma ) )V- \sigma \nabla _VV + (V^h(\delta)
\nonumber \\
&+(\nabla _VX)^v(\delta )) \nabla _VX+\delta \nabla _V \nabla _V X - \delta \nabla _{\nabla _VV}X \rbrace ^v
\nonumber \\
&-\frac{1}{2}g(V,V) \bar{\nabla}\alpha +g(V,\nabla _VX) \bar{\nabla} \sigma -\frac{1}{2} g(\nabla _VX,\nabla _VX)\bar{\nabla}\delta
\end{align}
where we used (11), (12), (13) and (14) in (22) and the equation $X_*(\nabla _V V)=(\nabla _V V)^h +(\nabla _{\nabla _V V}X)^v$.

So after long calculations we have the following theorem.

\begin{theorem} Let $(M,g)$ be a Riemannian manifold, not necessarily compact, and $X \in \Gamma TM$ be a tangent vector field on $M$ and $\lbrace  V_1,...,V_n\rbrace $ be a locally orthonormal basis for $TM$. The tension tensor field of $X$, i.e., $\tau \in \Gamma X^{-1} TT(M)$, is given by
\begin{align}
\tau (X)&=
\nonumber \\
&=\frac{1}{\alpha} \lbrace \pi _*(\bar{\nabla}\alpha ) + \sum _{i=1}^n [ (\nabla _{V_i} X)^v (\alpha ))V_i] - \sigma Ric(X)- \sum _{i=1}^n[(V_i^h(\sigma )
\nonumber \\
&+ (\nabla _{V_i} X)^v (\sigma ) \nabla _{V_i} X] - \sigma tr_g(\nabla _. \nabla _. X )+ \delta  tr_g R(X,\nabla _. X).  \rbrace ^h
\nonumber \\
&+\frac{1}{\delta} \lbrace - \sum _{i=1}^n[(V_i^h(\sigma ) +(\nabla _{V_i}X)^v (\sigma ) )V_i]- \sigma \sum _{i=1}^n \nabla _{V_i}V_i 
\nonumber \\
&+\sum _{i=1}^n [(V_i^h(\delta)+(\nabla _{V_i}X)^v(\delta )) \nabla _{V_i}X]+\delta \Delta _gX \rbrace ^v
\nonumber \\
&-\frac{n}{2}\bar{\nabla} \alpha +div(X) \bar{\nabla} \sigma - \frac{1}{2} g(\nabla X,\nabla X)\bar{\nabla}\delta 
\end{align}
\end{theorem}
\begin{proof}
Using (15) gives us
\begin{align}
\tau (X)&= \sum _{i=1}^n \beta _X(V_i ,V_i) 
\nonumber \\
&=\sum _{i=1}^n \big{\lbrace } \frac{1}{\alpha} \lbrace ( V_i^h(\alpha ) + (\nabla _{V_i }X)^v (\alpha ))V_i - \sigma R(X,V_i)V_i-(V_i^h(\sigma )
\nonumber \\
&+ (\nabla _{V_i} X)^v (\sigma )) \nabla _{V_i }X - \sigma \nabla _{V_i }\nabla _{V_i }X + \delta R(X,\nabla _{V_i }X)V_i  \rbrace ^h
\nonumber \\
&+\frac{1}{\delta} \lbrace -(V_i^h(\sigma ) +(\nabla _{V_i}X)^v (\sigma ) )V_i- \sigma \nabla _{V_i}V_i + (V_i^h(\delta)
\nonumber \\
&+(\nabla _{V_i}X)^v(\delta )) \nabla _{V_i}X+\delta \nabla _{V_i }\nabla _{V_i} X - \delta \nabla _{\nabla _{V_i}V_i}X \rbrace ^v
\nonumber \\
&-\frac{1}{2}g(V_i,V_i) \bar{\nabla}\alpha +g(V_i,\nabla _{V_i}X) \bar{\nabla} \sigma -\frac{1}{2} g(\nabla _{V_i}X,\nabla _{V_i}X)\bar{\nabla}\delta   \big{\rbrace} .
\end{align}
By substituting the following equations in (26) completes the proof. 
\begin{align*}
& \sum _{i=1}^n V_i^h (\alpha ) V_i^h =\sum _{i=1}^ng_{\delta ,\sigma}(\bar{\nabla} \alpha , V_i^h)V_i^h=h(\bar{\nabla} \alpha)\\
& \sum _{i=1}^n R(X,V_i)V_i=Ric(X)\\
& \sum _{i=1}^n \nabla _{V_i} \nabla _{V_i} X= tr_g \nabla _. \nabla _. X\\
&\sum _{i=1}^nR(X, \nabla _{V_i}X)V_i=tr_g R(X,\nabla _. X).\\
&\sum _{i=1}^n [ \nabla _{V_i} \nabla _{V_i}X - \nabla _{\nabla_{V_i} V_i}X ] =\Delta _g X
\end{align*}
\end{proof}
We end this section with a basic theorem in harmonic vector fields, which is very important to calculate the critical points of energy functional.

\begin{theorem} let $(M,g)$ be a compact orientable Riemannian manifold and $E: \Gamma (TM) \longrightarrow R^+$ the energy functional restricted to the space of all vector fields. Moreover, let $X_t :(-\varepsilon , + \varepsilon ) \longrightarrow TM $ be a variation along $X$ and $V$ be its variational vector field. If the tangent bundle of $M$ is equipped with an arbitrary Riemannian metric $G$, Then
\begin{align}
\frac{d}{dt} \lbrace E(X_t) \rbrace \mid _{t=0}=- \int _M g_{\delta , \sigma } (V^v,\tau (X)) \hspace{1mm} dvol(g) ,
\end{align}
where $X$ is supposed as a map from $(M,g)$ to $(TM,G)$.
\end{theorem}
\section{Tension tensor field of unit vector fields from $(M,g)$ to $(S(M),g_{\delta , \sigma} )$ and harmonicity of them}
The spherical bundle $S(M)$ on $(M,g)$ at every point $p \in M$ is
\begin{align*}
S_p(M)= \lbrace v \in T_pM | g(v,v)=1 \rbrace .
\end{align*}

The tangent space of $S(M)$ is $\mathcal{H} \oplus \bar{\mathcal{V}}$ where $\mathcal{H} $ is the horizontal sub-bundle of $TTM$ with respect to $\nabla$ and $\bar{\mathcal{V}}$ is the vector bundle on $S(M)$ defined by
\begin{align}
\bar{\mathcal{V}}_{(p,u)}= \lbrace Y_p^v | g(Y_p,u)=0 \rbrace = \lbrace Y_p^v-g(Y_p,u)u^v |Y_p \in T_pM  \rbrace \quad (p,u) \in TM.
\end{align}
Assuming that  $N_{(p,u)}^{g_{\delta , \sigma }}= \frac{\frac{\sigma}{\alpha}u^h+u^v}{\sqrt{\sigma - \frac{2 \sigma ^2 }{\alpha}+ \delta}}$ is a vector field on $TM$, one can simply derive $N_{(p,u)}^{g_{\delta , \sigma }}$ is normal unit vector field to $T_{(p,u)} S(M)$ with respect to $ g_{\delta ,\sigma}$.

We equip $S(M)$ with induced metric $i^*g_{\delta , \sigma }$, where $i:S(M) \longrightarrow TM$ is the inclusion map. From \cite{dragomir},   the tension tensor field of $ X:(M,g) \longrightarrow (S(M),i^* g_{\delta ,\sigma })$ is the tangent part of $\tau (X)$ with respect to $g_{\delta , \sigma }$, i.e., $\tau _1(X)=tan \hspace{1mm} \tau (X)$. So, we have the following proposition.

\begin{proposition} Let $X:(M,g) \longrightarrow (S(M),i^*g_{\delta , \sigma })$ be a unit vector field on $M$. Then the tension tensor field of $X$ is
\begin{align}
\tau _1(X)= \tau (X) - g_{\delta , \sigma } (\tau (X) ,  \frac{\frac{\sigma}{\alpha}X^h+X^v}{\sqrt{\sigma - \frac{2 \sigma ^2 }{\alpha}+ \delta}} )  \frac{\frac{\sigma}{\alpha}X^h+X^v}{\sqrt{\sigma - \frac{2 \sigma ^2 }{\alpha}+ \delta}} .
\end{align}
\end{proposition}
\begin{proof}
According to the definition of tension tensor field of $X$, $\tau _1(X)$ is tangent to spherical bundle at $X$, i.e.,  $\tau _1(X)\in T_X S(M)$. Since $ \frac{\frac{\sigma}{\alpha}X^h+X^v}{\sqrt{\sigma - \frac{2 \sigma ^2 }{\alpha}+ \delta}}$ is normal unit vector to $S(M)$ at $X$, we have
\begin{align*}
\tau _1(X)= \tau (X) - g_{\delta , \sigma } (\tau (X) ,  \frac{\frac{\sigma}{\alpha}X^h+X^v}{\sqrt{\sigma - \frac{2 \sigma ^2 }{\alpha}+ \delta}} )  \frac{\frac{\sigma}{\alpha}X^h+X^v}{\sqrt{\sigma - \frac{2 \sigma ^2 }{\alpha}+ \delta}} ,
\end{align*}
and the proof is completed.
\end{proof}

Hereafter, we assum that $\sigma =0$, and present the essential and sufficient conditions for harmonicity of unit vector field $X: (M,g) \longrightarrow (S(M), i^*g_{\delta ,0})$.

\begin{theorem}
If the horizontal and vertical sub-bundles are perpendicular to each other, i.e., $\sigma =0$, then the tension tensor field of unit vector field $X$ is given by
\begin{align}
\tau _1(X)&=\frac{1}{\alpha} \lbrace \pi _*(\bar{\nabla}\alpha ) + \sum _{i=1}^n [ (\nabla _{V_i} X)^v (\alpha ))V_i]  + \delta  tr_g R(X,\nabla _. X).  \rbrace ^h
\nonumber \\
&+\frac{1}{\delta} \lbrace \sum _{i=1}^n [(V_i^h(\delta)+(\nabla _{V_i}X)^v(\delta )) \nabla _{V_i}X]+\delta \Delta _gX 
\nonumber \\
& - [\delta g(\Delta _gX,X) -\frac{n}{2} d \alpha (X^v)-\frac{1}{2}g( \nabla X, \nabla X) d \delta (X^v)] X\rbrace ^v
\nonumber \\
&-\frac{n}{2}\bar{\nabla} \alpha  - \frac{1}{2} g(\nabla X,\nabla X)\bar{\nabla}\delta  ,
\end{align}
where $ \Delta _g X$ is the Laplacian \cite{dragomir} of $X$ defined by
\begin{align*}
- \Delta _g X= \sum _{i=1}^n [ \nabla _{V_i} \nabla _{V_i} X - \nabla _{\nabla _{V_i} V_i} X],
\end{align*}
and $\lbrace V_i \rbrace _{i=1} ^n$ is a locally orthonormal basis for $TM$.
\end{theorem}

\begin{proof}
Substituting  $\sigma =0$ in (25) gives us
\begin{align}
\tau (X) &= \frac{1}{\alpha} \lbrace \pi _*(\bar{\nabla}\alpha ) + \sum _{i=1}^n [ (\nabla _{V_i} X)^v (\alpha ))V_i]  + \delta  tr_g R(X,\nabla _. X).  \rbrace ^h
\nonumber \\
&+\frac{1}{\delta} \lbrace \sum _{i=1}^n [(V_i^h(\delta)+(\nabla _{V_i}X)^v(\delta )) \nabla _{V_i}X]+\delta \Delta _gX \rbrace ^v
\nonumber \\
&- \frac{n}{2} \bar{\nabla} \alpha  - \frac{1}{2} g(\nabla X,\nabla X)\bar{\nabla}\delta  .
\end{align}
Taking into account (31) and $\sigma =0$ in (29) gives us
\begin{align*}
\tau _1 (X)&=\frac{1}{\alpha} \lbrace \pi _*(\bar{\nabla}\alpha ) + \sum _{i=1}^n [ (\nabla _{V_i} X)^v (\alpha ))V_i]  + \delta  tr_g R(X,\nabla _. X).  \rbrace ^h\\
&+\frac{1}{\delta} \lbrace \sum _{i=1}^n [(V_i^h(\delta)+(\nabla _{V_i}X)^v(\delta )) \nabla _{V_i}X]+\delta \Delta _gX \rbrace ^v\\
& - g _{\delta , 0} (\frac{1}{ \delta} \lbrace \sum _{i=1}^n [(V_i^h(\delta)+(\nabla _{V_i}X)^v( \delta )) \nabla _{V_i}X]+\delta \Delta _gX \rbrace ^v \\
& -\frac{n}{2}\bar{\nabla} \alpha  - \frac{1}{2} g(\nabla X,\nabla X)\bar{\nabla}\delta  , X^v) \frac{X^v}{ \delta} - \frac{n}{2} \bar{\nabla} \alpha  - \frac{1}{2} g(\nabla X,\nabla X)\bar{\nabla}\delta .
\end{align*}
Using the definition of $g_{\delta ,0}$ and the fact that $g(\nabla _{ V_i}X , X)=0$ for every $i=1,...,n$, we have
\begin{align*}
\tau _1 (X)&=\frac{1}{\alpha} \lbrace \pi _*(\bar{\nabla}\alpha ) + \sum _{i=1}^n [ (\nabla _{V_i} X)^v (\alpha ))V_i]  + \delta  tr_g R(X,\nabla _. X).  \rbrace ^h\\
&+\frac{1}{\delta} \lbrace \sum _{i=1}^n [(V_i^h(\delta)+(\nabla _{V_i}X)^v(\delta )) \nabla _{V_i}X]+\delta \Delta _gX \rbrace ^v\\
&-\frac{1}{\delta}[g( \delta   \Delta _g X ,X) - \frac{n}{2} d \alpha (X^v) - \frac{1}{2} g(\nabla X , \nabla X) d \delta (X^v) ] X^v \\
&- \frac{n}{2} \bar{\nabla} \alpha  - \frac{1}{2} g(\nabla X,\nabla X)\bar{\nabla}\delta  ,
\end{align*}
and the proof is completed.
\end{proof}
From \cite{dragomir}, we know that a map between Riemanian manifolds is harmonic if and only if the tension tensor field of it, is zero. So, the vector field $X:(M,g) \longrightarrow (S(M),i^*g_{\delta , 0})$ is a harmonic map if and only if
\begin{align*}
&\frac{1}{\alpha} \lbrace \pi _*(\bar{\nabla}\alpha ) + \sum _{i=1}^n [ (\nabla _{V_i} X)^v (\alpha ))V_i]  + \delta  tr_g R(X,\nabla _. X). \rbrace \\
&- \pi _*(\frac{n}{2} \bar{\nabla} \alpha  + \frac{1}{2} g(\nabla X,\nabla X)\bar{\nabla}\delta )=0 ,\\
&and\\
&\frac{1}{\delta} \lbrace \sum _{i=1}^n [(V_i^h(\delta)+(\nabla _{V_i}X)^v(\delta )) \nabla _{V_i}X]+\delta \Delta _gX -g( \delta   \Delta _g X ,X)\\
& + \frac{n}{2} d \alpha (X^v) + \frac{1}{2} g(\nabla X , \nabla X) d \delta (X^v) ] X \rbrace -K( \frac{n}{2} \bar{\nabla} \alpha +  \frac{1}{2} g(\nabla X,\nabla X)\bar{\nabla}\delta)=0 .
\end{align*}

According to a theorem \cite{dragomir} in harmonic theory on vector fields, if $G$ is an arbitrary Riemannian metric on $S(M)$ and $V$ is an orthogonal vector field to $X:(M,g) \longrightarrow (S(M), G)$ with respect to $g$, one can get
\begin{align}
\int _M G(\tau _1 (X),V^v) dvol(g)=0.
\end{align}
Relation (32) will be used in proof of main theorem. Morever, if $\mathcal{V}$ is variational vector field of an arbitrary variation of $X$ through unit vector fields, then $g(K (\mathcal{V}),X)=0$.

The following lemma is a famous lemma in harmonic theory.
\begin{lemma} A unit vector field $X:(M,g) \longrightarrow (S(M),G)$ is harmonic if and only if 
\begin{align}
\frac{d}{dt} \lbrace E(\mathcal{U}_t) \rbrace \mid _{t=0}=- \int _M G (V^v,\tau _1 (X)) \hspace{1mm} dvol(g)=0,
\end{align}
where $\mathcal{U}_t$ is a variation along $X$ through unit vector fields and $\mathcal{V}=V^v$ is its variation vector field.
\end{lemma}
Main theorem of paper, specified as theorem 5.4. It presents condition on vector fields to be harmonic as a map from $(M,g)$ to $(S(M), g_{\delta ,0})$.
\begin{theorem}
Let $(M,g)$ be a compact orientable Riemannian manifold and $X$ be a unit vector field on $M$. Then $X:(M,g) \longrightarrow (S(M), g_{\delta ,0})$ is a harmonic vector field if and only if
\begin{align}
\Delta _gX &=\frac{1}{\delta}[(\delta - \frac{1}{2}d \delta (X^v)) ||\nabla X||^2 - \frac{n}{2} d \alpha (X^v)  ] X
\nonumber \\
& +\frac{n}{2}K( \bar{\nabla} \alpha ) + \frac{1}{2} ||\nabla X ||^2K( \bar{\nabla}\delta )
\nonumber \\
&-\frac{1}{\delta}  \sum _{i=1}^n [(V_i^h(\delta)+(\nabla _{V_i}X)^v(\delta )) \nabla _{V_i}X],
\end{align}
where $K$ is connection map with respect to the Levi-civita connection of $g$.
\end{theorem}

\begin{proof}
($\Longrightarrow $) Let $X$ be a harmonic vector field, we show that (34) holds. suppose,
\begin{align}
\tau _1(X)=\zeta X^h + \lambda X^v +V^v + W^h,
\end{align}
where $V$ and $W$ are perpendicular vector fields to $X$ and $\lambda ,  \zeta :X(M) \longrightarrow R $ are $C^{\infty}$ functions. we show that $V=0$ and $\lambda =0$. From (35), we have
\begin{align}
||V^v||^2=g_{\delta,0}( \tau _1(X),V^v).
\end{align}
According to (32), $ \int_M ||V^v||^2 dvol(g) = \int_M g_{\delta ,0}(\tau _1 (X) ,V^v) dvol(g)=0$. This shows that $V=0$, and $\tau _1(X)=\tau (X) - g_{\delta ,0}(\tau (X) , X^v) \frac{X^v}{\delta}$ shows that $\tau _1(X)$ hasn't any component in direction of $X^v$, i.e., $\lambda =0$. From (35) and $V=0$ and $\lambda =0$, one can get $K(\tau _1 (X))= 0$. On the other hand, from (30), we have
\begin{align}
K(\tau _1 (X))&=\frac{1}{\delta} \lbrace \sum _{i=1}^n [(V_i^h(\delta)+(\nabla _{V_i}X)^v(\delta )) \nabla _{V_i}X]+\delta \Delta _gX \rbrace 
\nonumber \\
&-\frac{1}{\delta}[g( \delta   \Delta _g X ,X) - \frac{n}{2} d \alpha (X^v) - \frac{1}{2} g(\nabla X , \nabla X) d \delta (X^v) ] X
\nonumber \\
&- \frac{n}{2}K( \bar{\nabla} \alpha ) - \frac{1}{2} g(\nabla X,\nabla X)K( \bar{\nabla}\delta ).
\end{align}
Using $g( \Delta _gX,X)=\frac{1}{2}\Delta (||X||^2) + ||\nabla X||^2$, we get
\begin{align}
K(\tau _1 (X))&=\frac{1}{\delta} \lbrace \sum _{i=1}^n [(V_i^h(\delta)+(\nabla _{V_i}X)^v(\delta )) \nabla _{V_i}X]+\delta \Delta _gX \rbrace 
\nonumber \\
&-\frac{1}{\delta}[(\delta - \frac{1}{2}d \delta (X^v)) ||\nabla X||^2 - \frac{n}{2} d \alpha (X^v)  ] X
\nonumber \\
&- \frac{n}{2}K( \bar{\nabla} \alpha ) - \frac{1}{2} g(\nabla X,\nabla X)K( \bar{\nabla}\delta ).
\end{align}
Using $K(\tau _1 (X))=0$ in (38) gives us
\begin{align}
\Delta _gX &=\frac{1}{\delta}[(\delta - \frac{1}{2}d \delta (X^v)) ||\nabla X||^2 - \frac{n}{2} d \alpha (X^v)  ] X
\nonumber \\
& +\frac{n}{2}K( \bar{\nabla} \alpha ) + \frac{1}{2} g(\nabla X,\nabla X)K( \bar{\nabla}\delta )
\nonumber \\
&-\frac{1}{\delta}  \sum _{i=1}^n [(V_i^h(\delta)+(\nabla _{V_i}X)^v(\delta )) \nabla _{V_i}X].
\end{align}
($\Longleftarrow$) Let (34) holds, we show that $X$ is a harmonic vector field. Substitute (34) in (37) gives us, $K(\tau _1 (X))=0$, i.e., the vertical part of $\tau _1(X)$ is zero. Lemma 5.3. with $K(\tau _1 (X))=0$ give 
\begin{align}
\frac{d}{dt} \lbrace E(\mathcal{U}_t) \rbrace \mid _{t=0}=- \int _M g_{\delta , 0} (V^v,\tau _1 (X)) \hspace{1mm} dvol(g)=0,
\end{align}
where $\mathcal{U}_t$ is a variation along $X$ through unit vector fields and $\mathcal{V}=V^v$ is its variation vector field. Note that the vertical and horizontal sub-bundles are perpendicular to each other with respect to $g_{\delta , 0}$.
\end{proof}
\begin{corollary}
Theorem 5.4. includes the particular theorem of Wiegmink \cite{wiegmink}.
\end{corollary}
\begin{proof}
Putting $\delta =1$ in Theorem 5.4. proves the corollary.
\end{proof}
The following proposition states that, if a unit vector field satisfies in (34) then it is a critical point of the following functional.
\begin{proposition}
The energy functional $E:\Gamma (S(M),g_{\delta ,0}) \longrightarrow R^+$ is given by
\begin{align}
E(X)=\frac{1}{2} \int _M (n \alpha +\delta ||\nabla X||^2) dvol(g),
\end{align}
for every, unit vector field $X:(M,g) \longrightarrow (S(M),g_{\delta , 0})$.
\begin{proof}
We know that the energy functional is given by
\begin{align}
E(X)=\frac{1}{2} \int _M ||dX||^2 dvol(g).
\end{align}
On the other hand
\begin{align}
||dX||^2&=tr_g X^* g_{\delta ,0}=\sum _{i=1}^n g_{\delta ,0}(X_* V_i , X_*V_i)
\nonumber \\
&=\sum _{i=1}^n g_{\delta ,0}(V_i^h +( \nabla _{V_i} X)^v ,V_i^h +( \nabla _{V_i} X)^v)
\nonumber \\
&=n \alpha + \delta || \nabla X||^2 ,
\end{align}
where $\lbrace V_1 ,...,V_n  \rbrace $ is an orthonormal locally basis for $TM$.
\end{proof}
\end{proposition}
We end our work with some examples on harmonic vector fields satisfing Theorem 5.4.
\begin{example}
Let $(R^n, \langle .,. \rangle )$ be Euclidean space and $(TR^n, g_s)$ be its tangent bundle equipped with Sasaki metric. Then, every parallel unit vector field is harmonic unit vector field. In this example, isotropic almost complex structure is the natural almost complex structure on $TR^n$.
\end{example}
\begin{example}
Define three complex structures $J_1,J_2$ and $J_3$ on Euclidean four space $(E^4 , \langle .,. \rangle) $ by
\begin{align}
&J_1(v_1,v_2,v_3,v_4)=(-v_2,v_1,-v_4,v_3),\\
&J_2(v_1,v_2,v_3,v_4)=(v_3,-v_4,-v_1,v_2),\\
&J_3(v_1,v_2,v_3,v_4)=(v_4,v_3,-v_2,-v_1),
\end{align}
where $ (v_1,v_2,v_3,v_4)$ is a tangent vector to $R^4$. It is simple to check that $(R^4,J_i , \langle .,. \rangle)$ for $i=1,2,3$ are kahler manifolds. Moreover, suppose $p=(p_1,p_2,p_3,p_4) $ be a point in $S^3$ and $N(p)=(p_1,p_2,p_3,p_4)$ be position vector field. Let $X_1(p)=J_1N, X_2(p)=J_2N$ and $X_3(p)=J_3N$ be tangent vector fields on $S^3$. One can check that $X_1,X_2$ and $X_3$ are unit vector fields and are perpendicular to each other, i.e., $\lbrace X_1, X_2 ,X_3 \rbrace$ is an orthonormal basis for $TS^3$.

We will show that $X_1:(S^3,g) \longrightarrow (S(S^3),g_{\delta ,0})$ is a harmonic unit vector field, That is, we will show that $X_1$ satisfies in (34) ($g$ is induced by Euclidean metric on $S^3$).

Using Gauss formula gives us the Levi-Civita connection $\nabla  ^S$ of $g$ as following
\begin{align}
 \nabla  ^S _Z W=\nabla ^E _ZW +g(Z,W)N.
\end{align}
If $V$ is a vector field on $S^3$ then from (47) one can get
\begin{align}
\nabla  ^S _V X_1&=\nabla ^E _VX_1 +g(V,X_1)N=\nabla ^E _VJ_1N +g(V,X_1)N
\nonumber \\
&=J_1V+g(V,X_1)N.
\end{align}
By the definition, the Laplacian of $X_1$ is 
\begin{align}
- \Delta _g X_1 &= \nabla ^S _{X_1} \nabla ^S _{X_1} X_1 + \nabla ^S _{X_2} \nabla ^S _{X_2} X_1 + \nabla ^S _{X_3} \nabla ^S _{X_3} X_1
\nonumber \\
& - \sum _{i=1} ^3 \nabla ^S _{\nabla ^S _{X_i}X_i} X_1 .
\end{align}
Using by (48) and (47) and the fact that ${\nabla ^S _{X_i}X_i}=0$ (because the integral curves of $X_i$ for all $i=1,2,3$ are geodesics of $S^3$), we have
\begin{align}
- \Delta _g X_1 &= J_1( \nabla ^E _{X_2} X_2 +\nabla ^E _{X_3} X_3).
\end{align}
With an straight forward calculation on (50), we have
\begin{align}
- \Delta _g X_1=-2J_1N=-2X_1,
\end{align}
so, $\Delta _g X_1=2X_1$. 

For calculating the right hand side of (34), we need to define $\alpha , \beta :TS^3  \longrightarrow R$. Let $V \in \Gamma (TS^3)$ be an arbitrary vector field on $S^3$. Since, $X_i$, $i=1,2,3  $ are global sections and basis on $TS^3$, we can write $V= \sum _{i=1}^3 V^i X_i $ where $V^i :TS^3 \longrightarrow R$ are smooth maps for $i=1,2,3$. We define $\alpha , \beta:TS^3 \longrightarrow R$ with definitions $\alpha (V)=\frac{(V^1)^2}{2}+1$ and $\beta (V) = \frac{1}{\alpha (V)}$. One can get
\begin{align}
&d\alpha (A)=V^1 dV^1(A) \quad and \hspace{6mm} (\bar{\nabla}\alpha)_V =V^1( \bar{\nabla} V^1)_V =\frac{V^1}{\delta}  (X_1^v)_V,\\
&d \delta (A)=-\delta ^2V^1 dV^1(A) \quad and \hspace{6mm} (\bar{\nabla}\delta)_V =-\delta  V^1 (X_1^v)_V.
\end{align}
Note that $\bar{\nabla} V^1=\frac{X_1^v}{\delta}$. Since in the right hand side of (34), $d\delta$ and $d\alpha $ are calculated at $X_1$ and also the vlues of $\alpha$ and $\delta$ are calculated at $X_1$, one can write
\begin{align}
&d\alpha |_{X_1} (X_1^v)=1  \quad and \hspace{6mm} (\bar{\nabla}\alpha)_{X_1} = (\frac{X_1^v}{\delta})_{X_1},\\
&d \delta |_{X_1}(X_1^v)=-\delta ^2 \quad and \hspace{6mm} (\bar{\nabla}\delta)_{X_1} =  (-\delta  X_1^v)_{X_1}.
\end{align}
Using (54) and (55), the right hand side of (34) is calculated as following
\begin{align}
&\frac{1}{\delta}[(\delta + \frac{1}{2} \delta ^2 ) 2 - \frac{n}{2}] X_1 +\frac{n}{2 \delta} X_1 - \delta  X_1
\nonumber \\
&=2X_1.
\end{align}
Note that the third line of (34) is zero. From (56) and (51) and (34) we conclude that $X_1$ is a harmonic unit vector field on $S^3$.
\end{example}
\section{Conclusios}
In this paper, the harmonicity of unit vector fields were investigated, whereas the tangent bundle was equipped with a Riemannian metric which is induced by isotropic almost complex structures, for this reason, we need to calculate the Levi-Civita connection of itroduced metric and tension tensor field of unit vector fields. The proof of Lemma 4.2 was an important step for calculating the tension tensor field of vector fields.

 In the Section 5, we calculated the tension tensor field of unit vector fields using by a helpful theorem which makes shortening the tension tensor field calculation. Since, For proving the main theorem we used the variational problem, Lemma 5.3 was most fundamental lemma to state.

 Finally we present some examples satisfying in the main theorem. The first example is a particular example, because we used the Sasaki metric on the tangent bundle and the second example is an important example, because of its connection with Hopf vector fields. This example leads us to investigate the harmonicity of Hopf vector fields in this sense. Future work is checking the harmonicity of Hopf vector fields on 3-sphere.

\vspace{.3cm}

\end{document}